\newif\ifpictures
\picturestrue

\documentclass[12pt]{amsart}
\usepackage{amssymb}
\usepackage{amsmath}
\usepackage[dvips]{graphicx}
\usepackage{book tabs}

\usepackage{dsfont}

\usepackage{float}

\usepackage[color]{changebar}
\cbcolor{blue}

\usepackage[color=green!05!red!05!blue!10]{todonotes}
\presetkeys{todonotes}{inline}{}

\ifpictures
\usepackage{psfrag}
\fi

\numberwithin{equation}{section}
\newtheorem{thm}{Theorem}
\newtheorem{prop}[thm]{Proposition}
\newtheorem{lemma}[thm]{Lemma}
\newtheorem{cor}[thm]{Corollary}

\theoremstyle{definition}

\newtheorem{example}[thm]{Example}
\newtheorem{remark1}[thm]{Remark}
\newtheorem{openproblem1}[thm]{Open problem}

\newtheorem*{ack}{Acknowledgement}

\newenvironment{rem}{\begin{remark1}\rm}{\end{remark1}}

\newenvironment{claim}{\normalfont{\bf Claim.}}{\normalfont}

\numberwithin{thm}{section}

\newcounter{FNC}[page]
\def\newfootnote#1{{\addtocounter{FNC}{2}$^\fnsymbol{FNC}$%
     \let\thefootnote\relax\footnotetext{$^\fnsymbol{FNC}$#1}}}

\newcommand{\N}{\mathbb{N}}
\newcommand{\Q}{\mathbb{Q}}
\newcommand{\R}{\mathbb{R}}

\newcommand{\VV}{\mathcal{V}}
\newcommand{\HH}{\mathcal{H}}
\newcommand{\Scal}{\mathcal{S}}

\DeclareMathOperator{\conv}{conv}

\DeclareMathOperator{\linspan}{span}
\DeclareMathOperator{\diag}{diag}
\DeclareMathOperator{\rank}{rank}

\usepackage{xcolor}

\title{Containment problems for polytopes and spectrahedra}

\author{Kai Kellner}

\author{Thorsten Theobald}

\author{Christian Trabandt}

\address{Goethe-Universit\"at, FB 12 -- Institut f\"ur Mathematik,
Postfach 11 19 32, D--60054 Frankfurt am Main, Germany}

\email{\{kellner,theobald,trabandt\}@math.uni-frankfurt.de}

\begin{document}

\begin{abstract}
We study the computational question whether a given polytope or spectrahedron
$S_A$ (as given by the positive semidefiniteness region of a linear matrix
pencil $A(x)$) is contained in another one $S_B$.
 
First we classify the computational complexity, extending results on the
polytope/poly\-tope-case by Gritzmann and Klee to the
polytope/spectrahedron-case. 
For various restricted containment problems, NP-hardness is shown.

We then study in detail semidefinite conditions to certify containment, building
upon work by Ben-Tal, Nemirovski and Helton, Klep, McCullough.
In particular, we discuss variations of a sufficient semidefinite condition to
certify containment of a spectrahedron in a spectrahedron.
It is shown that these sufficient conditions even provide exact semidefinite
characterizations for containment in several important cases, including
containment of a spectrahedron in a polyhedron. 
Moreover, in the case of bounded $S_A$ the criteria will always succeed in
certifying containment of some scaled spectrahedron $\nu S_A$ in $S_B$.
\end{abstract}

\maketitle

\section{Introduction}

Denote by $\mathcal{S}_k$ the set of all real symmetric $k \times k$-matrices
and by $\mathcal{S}_k[x]$ the set of symmetric $k \times k$-matrices with
polynomial entries in $x=(x_1, \ldots,x_n)$.
For $A_0, \ldots, A_n \in \mathcal{S}_k$, let $A(x)$ denote the \emph{linear
(matrix) pencil} 
$A(x) \ = \ A_0 + x_1 A_1 + \cdots + x_n A_n \in \mathcal{S}_k[x]$.
Then the set
\begin{equation}
\label{eq:spectrahedron1}
  S_A \ := \ \{x \in \R^n \, : \, A(x) \succeq 0 \}
\end{equation}
is called a \emph{spectrahedron}, where $A(x) \succeq 0$ denotes positive
semidefiniteness of the matrix $A(x)$.

Spectrahedra arise as feasible sets of semidefinite programming~(see
\cite{Ramana1995,Pataki2000}).
In the last years, there has been strong interest in understanding the geometry
of spectrahedra (see, e.g.,
\cite{barvinok-2012,gouveia-netzer-2011,helton-nie-2010}), particularly driven
by their intrinsic relevance in polynomial optimization
\cite{blekherman-2012,gouveia2010} and convex algebraic geometry
\cite{helton-nie-2012, helton-vinnikov-2007}.
Spectrahedra naturally generalize the class of polyhedra, see
\cite{Bhardwaj2011,Ramana1998} for particular connections between these two
classes.

In this paper, we study containment problems for polyhedra and spectrahedra.
Since polyhedra are special cases of spectrahedra, we can use the following
general setup: Given two linear pencils $A(x) \in \mathcal{S}_k[x]$ and 
$B(x) \in \mathcal{S}_l[x]$, is $S_A \subseteq S_B$?

For polytopes (i.e., bounded polyhedra), the computational geometry and
computational complexity of containment problems have been studied in detail.
See in particular the classifications by Gritzmann and Klee
\cite{gritzmann-klee-92,gritzmann-klee-93,gritzmann-klee-containment-survey}.
Notably, it is well-known that the computational complexity of deciding
containment problems strongly depends on the type of the input. 
For instance, if both polytopes are given by their vertices ($\VV$-polytopes),
or both polytopes are given as an intersection of halfspaces ($\HH$-polytopes), 
containment can be decided in polynomial time, while it is co-NP-hard to decide
whether an $\HH$-polytope is contained in a $\VV$-polytope (see
\cite{freund-orlin-85,gritzmann-klee-containment-survey}).

For spectrahedra, much less is known. 
Ben-Tal and Nemirovski studied the matrix cube problem
\cite{ben-tal-nemirovski-2002}, which corresponds to the containment problem
where $S_A$ is a cube.
In a much more general setting, Helton, Klep, and McCullough \cite{Helton2010}
studied containment problems of matricial positivity domains (which live in a
union of spaces of different dimensions). As a byproduct, they also derive some
implications for containment of spectrahedra.

In the current paper, we study containment problems of polytopes, polyhedra and
spectrahedra from a computational viewpoint. 
In Section~\ref{se:complexity}, we extend existing complexity classifications
for the polyhedral situation to the situation where polytopes and spectrahedra
are involved. 
In particular, the containment question of a $\VV$-polytope in a spectrahedron
can be decided in polynomial time, and the question whether a spectrahedron is
contained in an $\HH$-polytope can be formulated by the complement of
semidefinite feasibility problems (involving also strict inequalities). Roughly
speaking, the other cases are co-NP-hard. This includes the containment problem
of an $\HH$-polytope in a spectrahedron, already when the spectrahedron is a
ball. The complete classification is stated in 
Theorems~\ref{thm:complexityVinS}--\ref{thm:complexity3}.

To overcome the situation that the general containment problem for spectrahedra
is co-NP-hard, relaxation techniques are of particular interest.
Our point of departure in Section~\ref{sec:criterion} is the relaxation from
\cite{Helton2010} which provides a distinguished sufficient criterion for
containment of a spectrahedron $S_A$ in a spectrahedron $S_B$
(see~\eqref{eq:inclusion}).
We provide an elementary derivation of this semidefinite relaxation (as opposed
to the operator-theoretic techniques used there) and study the quality of the
criterion.
This leads to a new and systematic access to studying containment problems
of polyhedra and spectrahedra and provides several new (and partially
unexpected) results.

In particular, we obtain the following new results:

1. We discuss variations of containment criterion~\eqref{eq:inclusion}, which
   lead to improved numerical results, see Theorem~\ref{thm:Helton_Contain}, 
   Corollaries~\ref{cor:contain-positive} and \ref{cor:contain_symmetry} and
   Section~\ref{sec:non-exact}.

2. We exhibit several cases when the criteria are exact (see
   Theorem~\ref{thm:summaryexactness}).
   For some of the cases we can provide elementary proofs.
   The main case in Theorem~\ref{thm:summaryexactness}
   states that the sufficient criteria for the containment of spectrahedra
   in polyhedra (in normal form) are \emph{exact} characterizations. The proof
   of the statements is given in
   Section~\ref{sec:spectrahedron-polytope-inclusion}, by developing
   various properties of the containment criteria (transitivity, block
   diagonalization) and combining them with duality theory of semidefinite
   programming. The exactness of the spectrahedron-polyhedron-case 
   is particularly surprising, since a priori the criteria depend on the linear
   pencil representation of the spectrahedron.

3. In Section~\ref{sec:non-exact}, we extend the results from \cite{Helton2010}
   on cases, where the criteria are not exact. For a counterexample in 
   \cite{Helton2010} we exhibit the phenomenon that the containment criteria
   will at least succeed in certifying that a scaled version of the
   spectrahedron $S_A$ is contained in $S_B$. 

4. In Proposition~\ref{prop:scaling}, we show that in the case of bounded $S_A$
   there \emph{always} exists a scaling factor $\nu > 0$ such that for the
   scaled spectrahedron pair $(\nu S_A, S_B$) the criteria~\eqref{eq:inclusion} 
   and~\eqref{eq:inclusion_relaxed} hold.

We will close the paper by explaining some implications of the scaling result
on the optimization version of containment problems (as also relevant for the
computation of geometric radii of convex bodies, e.g., in 
\cite{gritzmann-klee-92,gritzmann-klee-93}).

\section{Preliminaries\label{se:prelim}}

Throughout the paper we work in $n$-dimensional Euclidean space $\R^n$, and
$\|\cdot\|$ denotes the Euclidean norm.

\subsection*{Matrices and block matrices}
For a matrix $A$, the $(i,j)$-th entry of $A$ is labeled by $a_{ij}$ as usual. 
For a block matrix $B$, we label the $(i,j)$-th block by $B_{ij}$ and the
$(s,t)$-th entry of $B_{ij}$ by $(B_{ij})_{st}$.

A square matrix with $1$ in the entry $(i,j)$ and zeros otherwise is denoted by
$E_{ij}$. The $n \times n$ identity matrix is denoted by $I_n$.

The \emph{Kronecker product} $A\otimes B$ of square matrices $A$ of size 
$k \times k$ and $B$ of size $l \times l$ is the $kl \times kl$ matrix
\begin{equation}
  A\otimes B =
  \begin{bmatrix} a_{11}\, B & \ldots & a_{1k}\, B \\
  \vdots & \ddots & \vdots \\
  a_{k1}\, B & \ldots & a_{kk}\, B
  \end{bmatrix}
\label{eq:kronecker}
\end{equation}
(see, e.g., \cite{de-klerk-book,Horn1994}).
It is well-known (see, e.g., \cite[Cor.\ 4.2.13]{Horn1994}) that the Kronecker
product of two positive semidefinite matrices is again positive semidefinite.

\subsection*{Polyhedra and polytopes}

A \emph{polyhedron} is the intersection of finitely many halfspaces. A bounded
polyhedron or, equivalently, the convex hull of finitely many points in $\R^n$
is called \emph{polytope}.

For algorithmic questions in $n$-dimensional space it is crucial whether a
polytope is given in the first way ($\mathcal{H}$-\emph{polytope}) or in the
second way ($\mathcal{V}$-\emph{polytope}).
Our model of computation is the binary Turing machine:
polytopes are presented by certain rational numbers, and the size of the input
is defined as the length of the binary encoding of the input data (see, e.g.,
\cite{gritzmann-klee-92}).
A $\mathcal{V}$-polytope $P$ is given by a tuple 
$(n; m; v^{(1)}, \ldots, v^{(m)})$ with $n,m \in \N$, and 
$v^{(1)}, \ldots, v^{(m)} \in \Q^n$ such that
$P = \mbox{conv}\{v^{(1)}, \ldots, v^{(m)}\}$.
An $\mathcal{H}$-polytope $P$ is given by a tuple $(n; m; A; b)$ with
$n,m \in \N$, a rational $m \times n$-matrix $A$, and $b \in \Q^m$ such that 
$P = \{x \in \R^n \, : \, b+Ax \ge 0\}$ is bounded. If the $i$-th
row $(b+Ax)_i \ge 0$ defines a facet of $P$, then the $i$-th row of $A$ is
an \emph{inner} normal vector of this facet.

For fixed dimension, $\mathcal{H}$- and $\mathcal{V}$-presentations of a
rational polytope can be converted into each other in polynomial time.
In general dimension (i.e., if the dimension is not fixed but part of the input)
the size of one presentation can be exponential in the size of the other
\cite{mcmullen-70}.

\subsection*{Spectrahedra}

Given a linear pencil 
\begin{equation}
\label{eq:linearpencil}
A(x) \ = \ A_0 + \sum_{p=1}^n x_p A_p \in \mathcal{S}_k[x] \quad
\text{ with } A_p = (a^p_{ij}) \, , \quad 0 \le p \le n \, ,
\end{equation}
the spectrahedron $S_A = \{x \in \R^n \, : \, A(x) \succeq 0\}$ contains the
origin in its interior if and only if there is another linear pencil $A'(x)$
with the same positivity domain such that $A'_0 = I$, see
\cite{Ramana1995,helton-vinnikov-2007}. 
In particular, then $S_A$ is full-dimensional.
To simplify notation, we sometimes assume that $A(x)$
is a \textit{monic} linear pencil, i.e. $A_0 = I_k$.
As a shorthand we use $A\succeq B$ to state that $A-B$ is positive semidefinite.

Note that every polyhedron $P = \{x \in \R^n \, : \, b + Ax \ge 0 \}$ has a 
natural representation as a spectrahedron:
\begin{equation}
  P = P_A = \left\{ x \in \R^n \ : \  A(x) = \begin{bmatrix} 
			  a_1(x)   & 0 & 0 \\ 
  			  0& \ddots & 0  \\
			  0 & 0  & a_k(x)\end{bmatrix}
			  \succeq 0 \right\},
\label{eq:polytope}
\end{equation}
where $a_i(x)$ abbreviates the $i$-th entry of the vector $b+Ax$. 
$P_A$ contains the origin if and only if the inequalities can be scaled so that
$b = \mathds{1}_k$, where $\mathds{1}_k$ denotes the all-ones vector in $\R^k$.
Hence, in this case, $A(x)$ is monic, and it is called the
\emph{normal form} of the polyhedron $P_A$.

A centrally-symmetric ellipsoid with axis-aligned semi-axes of lengths 
$a_1,\ldots, a_n$ can be written as the spectrahedron $S_A$ of the monic linear
pencil
\begin{equation}
  \label{eq:ellipsoid}
  A(x) \ = \ I_{n+1}
       + \sum_{p=1}^{n} \frac{x_p}{a_p}(E_{p,n+1}+E_{n+1,p}).
\end{equation}
We call~\eqref{eq:ellipsoid} the \emph{normal form of the ellipsoid}.
Specifically, for the case of equal semi-axis lengths $r := a_1 = \cdots = a_n$
this gives the \emph{normal form of a ball} with radius $r$.

For algorithmic questions, a linear pencil is given by a tuple 
$(n;k;A_0, \ldots, A_n)$ with $n,k\in \N$ and $A_0, \ldots, A_n$ rational
symmetric matrices.

\section{Complexity of containment problems for 
spectrahedra\label{se:complexity}}

In this section, we classify the complexity of several natural containment
problems for spectrahedra.
For polytopes the computational complexity of containment problems strongly
depends on the type of input representations. For $\mathcal{V}$- and
$\mathcal{H}$-presented polytopes, the following result is well-known (see
\cite{freund-orlin-85,gritzmann-klee-containment-survey}).

\begin{prop} \label{prop:complexity1}
Deciding whether a polytope $P$ is contained in a polytope $Q$ can be done in
polynomial time for the following cases:
\begin{enumerate}
\item Both $P$ and $Q$ are $\mathcal{H}$-polytopes,
\item both $P$ and $Q$ are $\mathcal{V}$-polytopes, or
\item $P$ is a $\mathcal{V}$-polytope while $Q$ is an $\mathcal{H}$-polytope.
\end{enumerate}
However, deciding whether an $\HH$-polytope is contained in a
$\mathcal{V}$-polytope is co-NP-complete. This hardness persists
if $P$ is restricted to be a standard cube and $Q$ is restricted
to be the affine image of a cross polytope.
\end{prop}

In the next statements we extend this classification to
containment problems involving polytopes and spectrahedra.
See~Table~\ref{tab:complexity-containment} for a summary.
Theorems~\ref{thm:complexityVinS} and~\ref{thm:complexitySinH}
give the positive results.

\begin{table}[H]
 \begin{tabular}{l|ccc}
  \toprule
   &\hspace*{0.2cm} $\HH$ & $\VV$ & $\Scal$ \\ 
  \midrule
  $\HH$ &\hspace*{0.2cm} P & co-NP-complete & co-NP-hard \\[+0.5ex]
  $\VV$ &\hspace*{0.2cm} P & P & P \\[+0.5ex]
  $\Scal$ &\hspace*{0.2cm} ``SDP'' & co-NP-hard & co-NP-hard \\
  \bottomrule
 \end{tabular}
\\[+0.5ex]
\caption{Computational complexity of containment problems, where the rows refer
to the inner set and the columns to the outer set and $\mathcal{S}$ abbreviates
spectrahedron.}
\label{tab:complexity-containment}
\end{table}

\begin{thm} \label{thm:complexityVinS}
Deciding whether a $\mathcal{V}$-polytope is contained in a spectrahedron can
be done in polynomial time.
\end{thm}

\begin{proof}
Given a $\VV$-presentation $P = \conv \{v^{(1)}, \ldots, v^{(m)}\}$ and a linear
matrix pencil $A(x)$, we have $P \subseteq S_A$ if and only if all the points
$v^{(i)}$ are contained in $S_A$. 
Thus, the containment problem is reduced to $m$ tests whether a certain
rational matrix is positive semidefinite. 
This can be decided in polynomial time, as one can compute, for a rational,
symmetric matrix $A$, a decomposition $A = U D U^T$ with a diagonal matrix $D$
in polynomial time (see, e.g., \cite{golub-van-loan-96}).
\end{proof}

Containment questions for spectrahedra are connected to feasibility questions
of semidefinite programs in a natural way.  A Semidefinite Feasibility Problem
(SDFP) is defined as the following decision problem (see, e.g.,
\cite{ramana1997}): 
Given a linear pencil defined by a tuple $(n;k;A_0, \ldots, A_n)$ with $n,k\in
\N$ and $A_0, \ldots, A_n$ rational symmetric matrices. Are there real numbers
$x_1, \ldots, x_n$ such that $A(x) = A_0 + \sum_{p=1}^n x_p A_p \succeq 0$, or
equivalently, is the spectrahedron $S_A$ non-empty?

Although semidefinite programs can be approximated up to an additive error of
$\varepsilon$ in polynomial time, the question 
``$\mathrm{SDFP} \in\mathrm{P}?$'' is one of the major open complexity
questions in semidefinite programming (see \cite{de-klerk-book, ramana1997}). 
Consequently, the following statement on containment of a spectrahedron in an
$\mathcal{H}$-polytope does not give a complete answer concerning polynomial
solvability of these containment questions in the Turing machine model. 
If the additional inequalities were non-strict, then we had to decide a finite
set of problems from the complement of the class SDFP.

\begin{thm} \label{thm:complexitySinH}
The problem of deciding whether a spectrahedron is contained in an 
$\mathcal{H}$-polytope can be formulated by the complement of semidefinite
feasibility problems (involving also strict inequalities), whose sizes are
polynomial in the description size of the input data.
\end{thm}

\begin{proof}
Let $A(x)$ be a linear matrix pencil and 
$P = \{x \in \R^n \, : \, b + Bx \ge 0\}$ with $B \in \Q^{m \times n}$ 
be an $\mathcal{H}$-polytope.
For each $i \in \{1, \ldots, m\}$ incorporate the linear condition
$ b_i + \sum_{j=1}^n b_{ij} x_j < 0$ into the linear pencil $A(x)$.
If one of the resulting $m$ (``semi-open'') spectrahedra is nonempty then 
$S_A \not\subseteq P$.
\end{proof}

The positive results in Theorems~\ref{thm:complexityVinS}
and~\ref{thm:complexitySinH} are contrasted by the following hardness results.

\begin{thm} \label{thm:complexity3} \
  
\begin{enumerate}
\item
  Deciding whether a spectrahedron is contained in a $\mathcal{V}$-polytope is
  co-NP-hard.
\item
  Deciding whether an $\mathcal{H}$-polytope or a spectrahedron is contained
  in a spectrahedron is co-NP-hard. This hardness statement persists if the
  $\mathcal{H}$-polytope is a standard cube or if the outer spectrahedron is a
  ball.
\end{enumerate}
\end{thm}

\begin{proof}
Deciding whether a spectrahedron $S_A$ is contained in a $\mathcal{V}$-polytope 
is co-NP-hard since already deciding whether an $\mathcal{H}$-polytope
is contained in a $\mathcal{V}$-polytope is co-NP-hard by
Proposition~\ref{prop:complexity1}.

Concerning the second statement,
co-NP-hardness of containment of $\mathcal{H}$-polytopes in spectrahedra
follows from Ben-Tal and Nemirovski
\cite[Proposition 4.1]{ben-tal-nemirovski-2002},
who use a reduction from the maximization of a positive semidefinite 
quadratic form over the unit cube. 

For the co-NP-hardness of containment of an $\mathcal{H}$-polytope in a ball,
we provide a reduction from
the NP-complete 3-satisfiability problem (3-SAT \cite{Cook1971}):
Does a given Boolean formula $\Phi$ over the variables $z_1, \ldots, z_n$ in
conjunctive normal form,  where each clause has at most 3 literals, admit an
assignment that evaluates \textsc{True}?

The $2^n$ possible assignments $\{\textsc{False}, \textsc{True}\}^n$ for $z_1,
\ldots, z_n$ can be identified with the vertices of an $n$-dimensional cube
$[-1,1]^n$. 
Let $B$ be a ball (which is a spectrahedron), such that the vertices of
$[-1,1]^n$ just ``peak'' through its boundary sphere $S$. 
Precisely (assuming w.l.o.g. $n \ge 2$), choose the radius $r$ of $B$ such that
\begin{equation*}
\left(\frac{1}{6}\right)^2 +
\left(\sqrt{n} -
\frac{1}{6}\right)^2 \ < \ r^2 \ < \ n \, .
\end{equation*}
Note that such a radius can be determined in polynomial time and size.

For the definition of the $\mathcal{H}$-polytope $P$, we start from the
$\mathcal{H}$-presentation
$\{x \in \R^n \, : \, -1 \le x_i \le 1 , \, 1 \le i \le n\}$ of $[-1,1]^n$ and
add one inequality for each clause of $\Phi$.
Let $\mathcal{C} = \mathcal{C}_1 \vee \cdots \vee \mathcal{C}_m$ be a 3-SAT
formula with clauses $\mathcal{C}_1, \ldots, \mathcal{C}_m$.
Denote by $\overline{z_i}$ the complement  of a variable $z_i$, and define the
literals $z_i^1 := z_i$, $z_i^0 := \overline{z_i}$. If the clause
$\mathcal{C}_i$ is of the form 
$\mathcal{C}_i = z_{i_1}^{e_{i_1}}\vee z_{i_2}^{e_{i_2}}\vee z_{i_3}^{e_{i_3}}$
with $e_{i_1}, e_{i_2}, e_{i_3} \in \{0,1\}$ then add the inequality
\[
  (-1)^{e_{i_1}} x_{i_1} + 
  (-1)^{e_{i_2}} x_{i_2} + 
  (-1)^{e_{i_3}} x_{i_3} \ \le \ 1 \, .
\]
If $P \subseteq B$ then, by the choice of $r$, none of the points in
$\{-1,1\}^n$
can be contained in $P$ and thus there does not exist a valid assignment for
$\Phi$. 
Conversely, assume that $P$ is not contained in $B$.
Let $p \in P \setminus B \subseteq [-1,1]^n$. We claim that componentwise
rounding
of $p$ yields an integer point $p' \in \{-1,1\}^n$  satisfying all defining
inequalities of
$P$. 
To see this, first note that by the choice of the radius of $B$, the components
$p_i$ of $p$ differ at most $\varepsilon < \frac{1}{3 \sqrt{2}} < \frac{1}{3}$
from
either $-1$ or $1$.

In order to inspect what happens to the inequalities when rounding, assume
without loss of generality that the inequality is of the form $x_1 + x_2 + x_3
\geq -1$.
We assume a rounded vector $p'$ does not satisfy the inequality, even though
$p$ does:
\begin{equation}
  p'_1 +p'_2 + p'_3  \ < \ -1  \, , \quad \text{but  }
  p_1 +p_2 + p_3 \ \geq \  -1 \label{eq:ineq1} \, .
\end{equation}
Since $p' \in \{-1,1\}^n$, \eqref{eq:ineq1} implies $p'_1 = p'_2 = p'_3 =-1$.
Hence, at least one of $p_1, p_2$ and $p_3$ differs from either $-1$ or $1$ by
more than $1/3$, which is a contradiction. This completes the reduction from
3-SAT.

Finally, deciding whether a spectrahedron $S_A$ is contained in a spectrahedron
$S_B$ is co-NP-hard, since already deciding whether an $\mathcal{H}$-polytope is
contained in a spectrahedron is co-NP-hard.
\end{proof}

\section{Relaxations and exact cases}
\label{sec:criterion}

In this section, we revisit and extend the relaxation techniques for the
containment of spectrahedra from~\cite{Helton2010}.
Our point of departure is the containment problem for pairs of $\HH$-polytopes,
which by Proposition~\ref{prop:complexity1} can be decided in polynomial time.
Indeed, this can be achieved by solving a linear program, as reviewed by the
following necessary and sufficient criterion.

\begin{prop} \label{prop:polytope_inclusion_sufficient}
Let $P_A = \{x \in \R^n \, : \, \mathds{1}_k + Ax \ge 0 \}$ and 
$P_B =\{x \in \R^n \, : \, \mathds{1}_l + Bx \ge 0 \}$ be polytopes.
There exists a right stochastic matrix $C$ (nonnegative entries, each row
summing to one) with $B = C A$ if and only if $P_A \subseteq P_B$.
\end{prop}

For preparing related statements in more general contexts below, we review the
proof which uses the following version of Farkas' Lemma:

\begin{prop}[Affine form of Farkas' Lemma {\cite[Corollary
7.1h]{Schrijver1986}}] \label{prop:farkas_aff}
Let the polyhedron $P=\{ x \in \R^n \ : \ l_i(x) \geq 0, i=1, \ldots, m\}$
with affine functions $l_i:\R^n \to \R$ be nonempty. 
Then every affine $l:\R^n \to \R$ that is nonnegative on $P$ can be written as 
$l(x) = c_0 + \sum_{i=1}^m c_i l_i(x)$ with nonnegative coefficients $c_i$. 
\end{prop}

\begin{proof} (of Proposition~\ref{prop:polytope_inclusion_sufficient}.)
If $B = CA$ with a right stochastic matrix $C$, then for any  $x \in P_A$ we
have $\mathds{1}_l + Bx = \mathds{1}_l + C (Ax) \ge 0$, i.e., $P_A \subseteq P_B$.

Conversely, if $P_A \subseteq P_B$, then for any $i \in \{1, \ldots, l\}$
the $i$-th row $(\mathds{1}_l + Bx)_i$ of $\mathds{1}_l + Bx$ is nonnegative on
$P_A$. Hence, by Proposition~\ref{prop:farkas_aff}, $(\mathds{1}_l + Bx)_i$ can
be written as a linear combination
\[
  (\mathds{1}_l + Bx)_i \ = \ 1 + (Bx)_i \ 
  = \ c'_{i0} + \sum_{j=1}^k c'_{ij} (\mathds{1}_k + Ax)_j
\]
with nonnegative coefficients $c'_{ij}$. Comparing coefficients yields
$\sum_{j=1}^k c'_{ij} = 1 - c'_{i0}$.
Since $P_A$ is a polytope with zero in its interior, the vertices of the polar
polytope $P_A^{\circ}$ are given by the rows $-A_j$ of $-A$.
Hence, for every $i \in \{1, \ldots, l\}$ there exists a convex combination 
$0 = \sum_{j=1}^k \lambda_{ij} (-A_j)$ with nonnegative $\lambda_{ij}$ and 
$\sum_{j=1}^k \lambda_{ij} = 1$, which we write as an identity 
$\sum_{j=1}^k \lambda_{ij} (\mathds{1}_k + A x)_j = 1$
of affine functions. By multiplying that equation with $c'_{i0}$, we
obtain nonnegative $c''_{ij}$ with 
$\sum_{j=1}^k c''_{ij} (\mathds{1}_k + Ax)_j = c'_{i0}$, which yields
\[
  1 + (Bx)_i \ = \ \sum_{j=1}^k (c'_{ij} + c''_{ij}) (\mathds{1}_k + Ax)_j .
\]
Hence, $C = (c_{ij})$ with $c_{ij} := c'_{ij} + c''_{ij}$ is a right stochastic
matrix with $B = CA$. 
\end{proof}

The sufficiency part from Proposition~\ref{prop:polytope_inclusion_sufficient}
can be extended to the case of spectrahedra in a natural way. The natural 
description of a polytope $P$ as a spectrahedron, as introduced in 
Section~\ref{se:prelim}, is given by
\[
  P \ = \ P_A \ = \ \left\{ x \in \R^n \ : \ 
  A(x) = \diag(a_1(x), \ldots, a_k(x)) \succeq 0 \right\},
\]
where $a_i(x)$ is the $i$-th entry of the vector $\mathds{1}_k + Ax$.
Then, as in the definition of a linear pencil \eqref{eq:linearpencil}, $A_p$ is
the $k\times k$ diagonal matrix $\diag(A_{:,\,p})$ of the $p$-th column of $A$.
Proceed in the same way with $P_B$. 
Now define a $kl\times kl$ matrix $C'$ by writing the entries of $C$ on the
diagonal, i.e. $C' =
\diag(c_{11},\ldots,c_{l1},c_{12},\ldots,c_{l2},\ldots,c_{1k},\ldots,c_{lk})$.
Then the condition from Proposition~\ref{prop:polytope_inclusion_sufficient}
translates to
\begin{equation}\label{eq:polytope-inclusion}
  C' \text{ diagonal, }
  C'=\left( C'_{ij} \right)_{i,j=1}^{k}\succeq 0,
  \quad I_{l} = \sum_{i=1}^{k} C'_{ii},
  \quad \forall p=1,\ldots,n:\ B_p 
  = \sum_{i=1}^k a^{p}_{ii} C'_{ii}\, ,
\end{equation}
where $a^p_{ij}$ is the $(i,j)$-th entry of $A_p$ and $C'_{ij}\in\R^{l\times
l}$.
Theorem~\ref{thm:Helton_Contain} below tells us, that $C'$ does not need to be
diagonal and yields a sufficient condition for the containment of spectrahedra. 

\subsection{A sufficient condition for containment of a spectrahedron in a
spectrahedron}
\label{sec:sufficient}

Let $A(x) \in \mathcal{S}_k[x]$ and $B(x) \in \mathcal{S}_l[x]$ be linear pencils.

In the following, the indeterminate matrix $C=\left(C_{ij}\right)_{i,j=1}^{k}$ 
(``\emph{Choi matrix}'') is a symmetric $kl\times kl$-matrix where the $C_{ij}$
are $l\times l$-blocks.
By showing the equivalence of containment of the so-called matricial relaxations
of two spectrahedra $S_{A},\ S_{B}$ given by monic linear pencils and the
existence of a completely positive unital linear map 
$ \tau :\ \linspan\{A_0,A_1,\ldots,A_n\} \rightarrow
  \linspan\{B_0,B_1,\ldots,B_n\},\ A_p\mapsto B_p $, 
the authors of \cite{Helton2012, Helton2010} proved that the system
\begin{equation}\label{eq:inclusion}
  C = \left( C_{ij} \right)_{i,j=1}^{k}\succeq 0,
  \quad \forall p = 0,\ldots,n:\ B_{p} 
  = \sum_{i,j=1}^{k} a^{p}_{ij} C_{ij}
\end{equation}
has a solution if and only if the matricial relaxation of $S_A$ is contained in
the one of $S_B$. If so, then $S_A\subseteq S_B$. We show the latter statement 
in an elementary way, see Theorem~\ref{thm:Helton_Contain}.

Moreover, in our approach it becomes apparent that we can relax the criterion
given by Helton, Klep and McCullough by replacing the linear constraint on the
constant matrices in~\eqref{eq:inclusion} with semidefinite constraints,
\begin{equation}\label{eq:inclusion_relaxed}
  C = \left( C_{ij} \right)_{i,j=1}^{k}\succeq 0, \quad
  B_0 - \sum_{i,j=1}^{k} a^0_{ij} C_{ij} \succeq 0, \quad 
  \forall p = 1,\ldots,n:\ B_{p} 
  = \sum_{i,j=1}^{k} a^{p}_{ij} C_{ij} .
\end{equation}

This relaxed system is still sufficient for containment of spectrahedra as the
following theorem shows.

\begin{thm} \label{thm:Helton_Contain}
Let $A(x) \in \mathcal{S}_k[x]$ and $B(x) \in \mathcal{S}_l[x]$ be linear
pencils. 
If one of the systems~\eqref{eq:inclusion} or~\eqref{eq:inclusion_relaxed} is
feasible then $S_A\subseteq S_B$.
\end{thm}

\begin{proof}
First we show the statement for~\eqref{eq:inclusion_relaxed}.

For $x\in S_A$, the last two conditions in~\eqref{eq:inclusion_relaxed} imply
\begin{align} \label{eq:containment_proof}
  B(x) &= B_0 + \sum_{p=1}^n x_{p} B_{p}\ \succeq \ \sum_{i,j=1}^{k} \,
a^{0}_{ij}\, C_{ij}
    + \sum_{p=1}^n \sum_{i,j=1}^{k} x_{p}\, a^{p}_{ij}\, C_{ij} 
   \ =\ \sum_{i,j=1}^{k} \left(A(x)\right)_{ij} C_{ij} \, .
\end{align}

Since $A(x)$ and $C$ are positive semidefinite, the Kronecker product
$A(x)\otimes C$ is positive semidefinite as well, see~\eqref{eq:kronecker}. As a
consequence, all principal submatrices of $A(x)\otimes C$ are positive
semidefinite. Consider the principal submatrix where we take the $(i,j)$-th
sub-block of every $(i,j)$-th block, 
\[
  \left( \left( A(x) \right)_{ij} C_{ij} \right)_{i,j=1}^{k}\in
 \mathcal{S}_{kl}[x] .
\]
To be more precise, $A(x)\otimes C$ is a $k^2 l\times k^2 l$-matrix with
$k\times k$ blocks of the form
\[
  (A(x))_{ij} C = 
  \begin{bmatrix}
    (A(x))_{ij} C_{11} & \cdots & (A(x))_{ij} C_{1k} \\
    \vdots & (A(x))_{ij} C_{ij} & \vdots \\[+0.5ex]
    (A(x))_{ij} C_{k1} & \cdots & (A(x))_{ij} C_{kk}
  \end{bmatrix}
  \in\mathcal{S}_{kl} .
\]
(Remember that $(A(x))_{ij}$ is a scalar). Now we take the
$(i,j)$-th block of $(A(x))_{ij} C$, i.e. $(A(x))_{ij} C_{ij}$.

Set $\mathbb{I} = \left[ I_l, \ldots, I_l \right]^T $. Then the claim for
system~\eqref{eq:inclusion_relaxed} follows from the
fact that positive semidefiniteness ``$\succeq$`` is a
transitive relation on the space of symmetric matrices, that is,
\begin{align}
  v^T B(x) v
  & \ge v^T\left( \mathbb{I}^T\,\left( \left( A(x) \right)_{ij}
     C_{ij} \right)_{i,j=1}^{k}\,\mathbb{I} \right) v \label{eq:le} \\
  &= \left( v^T\ \ldots\ v^T \right) \left( \left( A(x) \right)_{ij}
     C_{ij} \right)_{i,j=1}^{k}\left(v\ \ldots\ v \right)^T \nonumber
  \geq 0
\end{align}
for every $v\in\R^{l}$.

Specializing ``$\succeq$'' to ``$=$'' in~\eqref{eq:containment_proof} 
and ``$\geq$`` to ``$=$`` in~\eqref{eq:le} provides a streamlined proof
for~\eqref{eq:inclusion}.
\end{proof}

For both systems~\eqref{eq:inclusion} and~\eqref{eq:inclusion_relaxed} the
feasibility depends on the linear pencil representation of the sets involved. In
Section~\ref{sec:non-exact} we will take a closer look at this fact. 

\begin{rem}
The sub-block argument in the proof of Theorem~\ref{thm:Helton_Contain} can also
be stated in terms of the Khatri-Rao product (see \cite{Liu-2000}).
Let $A = (A_{ij})_{ij}$ and $B = (B_{ij})_{ij}$ be block matrices, consisting
of $k\times k$ blocks of size $p \times p$ and $q \times q$, respectively. 
The \emph{Khatri-Rao product} of $A$ and $B$ is defined as the blockwise
Kronecker product of $A$ and $B$, i.e.,
\[
  A \ast B = \left( A_{ij} \otimes B_{ij} \right)_{ij} .
\]
If both $A$ and $B$ are positive semidefinite, then the Khatri-Rao product
$A\ast B$ is positive semidefinite as well, see \cite[Theorem 5]{Liu-2000}.

Now consider $A(x)$ and $C$ as in the proof of Theorem~\ref{thm:Helton_Contain}.
Then $p=1$ and $q=l$. Therefore, 
\[
  A(x) \ast C = \left( (A(x))_{ij} \otimes C_{ij} \right)_{i,j=1}^k
	      = \left( (A(x))_{ij} C_{ij} \right)_{i,j=1}^k
\]
is positive semidefinite.
\end{rem}

The subsequent statement shows that Theorem~\ref{thm:Helton_Contain} is
invariant under translation. Let $S_A$ be a spectrahedron defined by the linear
pencil $A(x) = A_0 + \sum_{p=1}^n x_p A_p$. To move $S_A$ by a vector $v = (v_1,
\ldots, v_n)$ we substitute $x-v$ into the pencil and get
\[
  A(x-v) = A_0 -\sum_{p=1}^n v_p A_p + \sum_{p=1}^n x_p A_p.
\]

\begin{lemma}[Translation symmetry]
The criteria \eqref{eq:inclusion} and \eqref{eq:inclusion_relaxed} are invariant
under translation. 
\end{lemma}

\begin{proof}
Given linear pencils $A(x)$ and $B(x)$, let $C$ be a solution to
system~\eqref{eq:inclusion_relaxed}. Then it is also a solution for the
translated pencils $A(x-v)$ and $B(x-v)$ for any $v\in\R^n$. The translation
only has an impact on the constant matrix, we have to show 
\begin{equation}
  B_0 - \sum_{p=1}^n v_p B_p 
  - \left( \sum_{i,j=1}^{k} \left( a^0_{ij} 
  - \sum_{p=1}^n v_p a^p_{ij} \right) C_{ij} \right) \succeq 0 .
\label{eq:sym_proof}
\end{equation}
Since $B_{p} = \sum_{i,j=1}^{k} a^{p}_{ij} C_{ij}$ for all $p = 1,\ldots,n$,
\eqref{eq:sym_proof} is equivalent to  $B_0 - \sum_{i,j=1}^{k} a^0_{ij} C_{ij}
\succeq 0$, which is the condition on the constant matrices before translating.

As in the proof of Theorem~\ref{thm:Helton_Contain}, specializing ``$\succeq$''
to ``$=$'' yields a proof for~\eqref{eq:inclusion}.
\end{proof}

If $S_B$ is contained in the positive orthant, we can give a stronger version of
the criterion introduced in Theorem~\ref{thm:Helton_Contain}.

\begin{cor}\label{cor:contain-positive}
Let $A(x) \in \mathcal{S}_k[x]$ and $B(x) \in \mathcal{S}_l[x]$ be linear
pencils and let $S_A$ be contained in the positive orthant.
If the following system is feasible then $S_A\subseteq S_B$.
\begin{equation}\label{eq:inclusion_positive}
  C = \left( C_{ij} \right)_{i,j=1}^{k}\succeq 0, \quad
  B_0 - \sum_{i,j=1}^{k} a^0_{ij} C_{ij} \succeq 0, \quad 
  \forall p = 1,\ldots,n:\ B_p - \sum_{i,j=1}^{k} a^p_{ij} C_{ij} \succeq 0 .
\end{equation}
\end{cor}

\begin{proof}
The proof is along the lines of the proof of Theorem~\ref{thm:Helton_Contain}.
Indeed, since $S_A$ lies in the positive orthant, we have $x\geq 0$ for all
$x\in S_A$ and hence, 
\begin{align*}
  B(x) &= B_0 + \sum_{p=1}^n x_{p} B_{p}\ 
  \succeq \ \sum_{i,j=1}^{k} \, a^{0}_{ij}\, C_{ij}
    + \sum_{p=1}^n \sum_{i,j=1}^{k} x_{p}\, a^{p}_{ij}\, C_{ij} 
   \ =\ \sum_{i,j=1}^{k} \left(A(x)\right)_{ij} C_{ij} \, .
\end{align*}
\end{proof}

By relaxing system~\eqref{eq:inclusion} to \eqref{eq:inclusion_positive} the
number of scalar variables remains $\frac{1}{2} kl(kl+1)$, whereas the
$\frac{1}{2}(n+1)l(l+1)$ linear constraints are replaced by $n+1$ semidefinite
constraints of size $l\times l$.

If containment restricted to the positive orthant implies containment everywhere
else, criterion~\eqref{eq:inclusion_positive} can be applied, even if the
spectrahedron is not completely contained in the positive orthant. To make use
of this fact, we have to premise a certain structure of the spectrahedra.
We give an example in the following corollary.

\begin{cor} \label{cor:contain_symmetry}
Let $A(x) \in \mathcal{S}_k[x]$ and $B(x) \in \mathcal{S}_l[x]$ be linear
pencils defining spectrahedra with a reflection symmetry with respect to all
coordinate hyperplanes. If system~\eqref{eq:inclusion_positive} is feasible then
$S_A\subseteq S_B$.
\end{cor}

In Section~\ref{sec:non-exact} we will see that the relaxed
version~\eqref{eq:inclusion_positive} is strictly stronger than
system~\eqref{eq:inclusion}. There are cases, where a solution to the relaxed
problem~\eqref{eq:inclusion_positive} exists, even though the original
problem~\eqref{eq:inclusion} is infeasible.

\subsection{Exact cases\label{se:exactcases}}

It turns out that the sufficient semidefinite criteria~\eqref{eq:inclusion}
and~\eqref{eq:inclusion_relaxed} even provide exact containment
characterizations in several important cases. 
Detailed statements of these results and their proofs will be given in
Statements~\ref{lem:polytope_inclusion_neccessary},
\ref{lem:ellipsoid_inclusion}, \ref{lem:ball-polytope_inclusion},
\ref{thm:spec_polytope}.

For ease of notation, most statements in this section as well as in
section~\ref{sec:spectrahedron-polytope-inclusion} are given for monic pencils
and proved only for criterion~\eqref{eq:inclusion}. Note however, that
feasibility of~\eqref{eq:inclusion} implies feasibility
of~\eqref{eq:inclusion_relaxed}. Furthermore, after translating the (in this
section mostly bounded) spectrahedra to the positive orthant,
Corollary~\ref{cor:contain-positive} can be applied. Since
criterion~\eqref{eq:inclusion} is invariant under translation, its feasibility
again implies that system~\eqref{eq:inclusion_positive} has a solution for the 
translated spectrahedra.

Besides the normal forms of polyhedra, ellipsoids, and balls introduced in
Section~\ref{se:prelim}, the exact characterizations will also use the following
\emph{extended form} $S_{\widehat{A}}$ of a spectrahedron~$S_A$. Given a linear
pencil $A(x)\in\mathcal{S}_k[x]$, we call the linear pencil with an additional
1 on the diagonal
\begin{equation}
  \widehat{A}(x) := \begin{bmatrix}
              1 & 0 \\
	      0 & A(x)
            \end{bmatrix} \in \mathcal{S}_{k+1}[x]
\label{eq:extended_spec}
\end{equation}
the {\it extended linear pencil} of $S_A = S_{\widehat{A}}$ (the spectrahedra
coincide, only the representations of $S_A$ and $S_{\widehat{A}}$ differ, since
the 1 we add for technical reasons is redundant). The entries of $\widehat{A}_p$ 
in the pencil $\widehat{A}(x) = \widehat{A}_0 + \sum_{p=1}^n x_p \widehat{A}_p$
are denoted by $\widehat{a}^{\,p}_{ij}$ for $i,j=0,\ldots,k$, as usual.

\begin{thm} \label{thm:summaryexactness}
Let $A(x) \in \mathcal{S}_k[x]$ and $B(x) \in \mathcal{S}_l[x]$ be monic linear
pencils.
In the following cases the criteria~\eqref{eq:inclusion} as well
as~\eqref{eq:inclusion_relaxed} are necessary and sufficient for the inclusion
$S_A \subseteq S_B$:
\begin{enumerate}
  \item if $A(x)$ and $B(x)$ are normal forms of ellipsoids
    (both centrally symmetric, axis-aligned semi-axes),
  \item if $A(x)$ and $B(x)$ are normal forms of a ball 
    and an $\HH$-polyhedron, respectively, 
  \item if $B(x)$ is the normal form of a polytope,
  \item if $\widehat{A}(x)$ is the extended form of a spectrahedron
    and $B(x)$ is the normal form of a polyhedron.
\end{enumerate}
\end{thm}

In this section, we provide the proofs of (1), (2), where the sufficiency parts
follow by Theorem~\ref{thm:Helton_Contain}. The cases~(3) and~(4) will be
treated in Section~\ref{sec:spectrahedron-polytope-inclusion}.
We start with the containment of $\HH$-polyhedra in $\HH$-polyhedra which
slightly generalizes Proposition~\ref{prop:polytope_inclusion_sufficient} and
will be used in the proofs of later statements.

\begin{lemma} \label{lem:polytope_inclusion_neccessary} 
Let $A(x)\in\mathcal{S}_k[x]$ be the normal form of a polyhedron as defined in
\eqref{eq:polytope} and let 
$\widehat{A}(x) = \diag(1, a_1(x), \ldots, a_k(x)) \in\mathcal{S}_{k+1}[x]$
be the extended linear pencil of $A(x)$. 
Let $B(x)\in\mathcal{S}_l[x]$ be the normal form of a polyhedron.
When applied to the pencils $\widehat{A}(x)$ and $B(x)$
criterion~\eqref{eq:inclusion} is necessary and sufficient for the inclusion
$S_{\widehat{A}} = S_A \subseteq S_B$. 
If $S_A$ is a polytope, i.e. a bounded polyhedron, the pencil $A(x)$ can be used
instead of $\widehat{A}(x)$.
\end{lemma}

\begin{proof} 
With regard to~\eqref{eq:polytope}, the polyhedra $S_A$ and $S_B$ are of the
form $S_A = \{ x \in \R^n \, : \, \mathds{1}_k + Ax \ge 0\}$ and 
$S_B = \{ x \in \R^n \, : \, \mathds{1}_l + Bx \ge 0\}$, respectively, and let 
$\widehat{A}$ be the $(k+1) \times n$ matrix defined by $\widehat{A} := $
\begin{footnotesize}$\Big[ \begin{array}{c} 0 \\ A \end{array} \Big]$
\end{footnotesize}.

If $S_{\widehat{A}} \subseteq S_B$, then there exist convex combinations
$(\mathds{1}_l + Bx)_i = c_{i0} + \sum_{j=1}^{k} c_{ij} (\mathds{1}_{k} + Ax)_j
= \sum_{j=0}^{k} c_{ij} (\mathds{1}_{k+1} + \widehat{A} x)_j,$ where
$(\mathds{1}_{k+1} + \widehat{A}x)_0 = 1$
and the coefficients $c_{ij}$ are nonnegative, just as in the proof of
Proposition~\ref{prop:polytope_inclusion_sufficient}.

Now we construct a matrix $C$ that is a solution to system~\eqref{eq:inclusion}.
Recall that $C$ consists of matrices of size 
$l \times l:\ C = (C_{st})_{s,t=0}^k$. 
Set the $i$-th diagonal entry of $C_{jj}$ to be $c_{ij}$, and choose all other
entries to be zero. The resulting matrix is a diagonal matrix with non-negative
entries, which makes it positive semidefinite.

Clearly, 
$ B(x) = \sum_{j=0}^k (\mathds{1}_{k+1} + \widehat{A}x)_j \, C_{jj}
  = \sum_{i,j=0}^k (\mathds{1}_{k+1} + \widehat{A}x)_j \, C_{ij}$. 
Comparing coefficients, we see that $ I_l = \sum_{j=0}^k C_{jj}$ and 
$B_p = \sum_{i,j=0}^k \widehat{a}^{\, p}_{ij} C_{ij}$ 
for all $p \in \{1,\ldots,n\}$.

If the inner polyhedron $S_A$ is a polytope, the constant $1$ is a convex
combination of the remaining polynomials $a_1(x), \ldots, a_k(x)$. Thus the
additional $1$ in the upper left entry of pencil $\widehat{A}(x)$ is not needed.
\end{proof}

As we have seen in the proof of Lemma~\ref{lem:polytope_inclusion_neccessary}, 
for polyhedra there is a diagonal solution to \eqref{eq:inclusion}. Thus it is
sufficient to check the feasibility of the restriction of \eqref{eq:inclusion}
to the diagonal and checking inclusion of polyhedra reduces to a linear program.

\begin{rem}
For unbounded polyhedra, the extended normal form is required in order for the
criterion to be exact. Without it, already in the simple case of two half spaces
defined by two parallel hyperplanes, system~\eqref{eq:inclusion} is not
feasible.
\end{rem}

The following statement on ellipsoids uses the normal form~\eqref{eq:ellipsoid}.

\begin{lemma}\label{lem:ellipsoid_inclusion}
Let two ellipsoids $S_A$ and $S_B$ be centered at the origin with semi-axes
parallel to the coordinate axes, given by the normal forms
\[
  A(x) = I_{n+1}
       + \sum_{p=1}^{n} \frac{x_{p}}{a_p} (E_{p,n+1}+E_{n+1,p}) 
  \text{ and }
  B(x) = I_{n+1}
       + \sum_{p=1}^{n} \frac{x_{p}}{b_p} (E_{p,n+1}+E_{n+1,p}) \, ,
\]
respectively. Here $\left( a_{1},\ldots,a_{n} \right) > 0$ and 
$\left( b_{1},\ldots,b_{n} \right) > 0$ are the vectors of the lengths of the 
semi-axes. Then system~\eqref{eq:inclusion} is necessary and sufficient for the 
inclusion $S_A \subseteq S_B$.
\end{lemma}

\begin{proof}
Note first that $k=l=n+1$. It is obvious that $S_A\subseteq S_B$ if and only if
$b_p - a_p \geq 0$ for every $p=1,\ldots,n$. 
The matrices underlying the matrix pencils $A(x)$ and $B(x)$ are
\[
  A_p = \frac{1}{a_{p}} (E_{p,n+1}+E_{n+1,p}) 
  \text{ and }
  B_p = \frac{1}{b_{p}} (E_{p,n+1}+E_{n+1,p})  
\]
for all $p=1,\ldots,n$. Now define an $(n+1)^2\times (n+1)^2$-block matrix $C$ by
\[
  \left(C_{i,j}\right)_{s,t} = 
\begin{cases}
  1 & i=j=s=t, \\
  \frac{a_{j}}{b_{j}} & i=s=n+1,\ j=t\leq n, \\
  \frac{a_{i}}{b_{i}} & i=s\leq n,\ j=t=n+1, \\
  \frac{a_i\, a_j}{b_i\, b_j} & i=s\leq n,\ j=t\leq n,\ i\neq j, \\
  0 & \text{ otherwise.}
\end{cases}
\]
We show that $C$ is a solution to~\eqref{eq:inclusion}. 
Decompose $x\in\mathbb{R}^{(n+1)^2}$ in blocks of length $n+1$ and write
$x_{i,j}$ for the $j$-th entry in the $i$-th block. The matrix $C$ is positive
semidefinite since
\begin{align*}
  x^T C x & = \sum_{i=1}^{n+1} x_{i,i}^2 
	    + 2 \sum_{i < j \leq n} \frac{ a_i\, a_j }{ b_i\, b_j }
	      x_{i,i} x_{j,j}
	    + 2 \sum_{i=1}^n\frac{ a_{i} }{ b_{i} }\, x_{i,i}\, x_{n+1,n+1}\\
	  & = \left( \sum_{i=1}^n \frac{a_i}{b_i}\, x_{i,i}
	    + x_{n+1,n+1} \right)^2
	    + \sum_{i=1}^n \left( 1 - \frac{a_{i}^2}{b_{i}^2} \right)
	      x_{i,i}^2 \geq 0
\end{align*}
for all $x\in\R^{(n+1)^2}$.
Clearly, the sum of the diagonal blocks is the identity matrix $I_{n+1}$. Since
every $A_p$ has only two non-zero entries, every $B_p$ is a linear combination
of only two blocks of $C$,
\[
  B_p = \frac{1}{a_p} C_{n+1,p} + \frac{1}{a_p} C_{p,n+1}.
\]
This equality is true by the definition of $C$.
\end{proof}

\begin{rem}
Using the square matrices $E_{ij}$ of size $(n+1)\times(n+1)$ introduced in
Section~\ref{se:prelim}, the matrix $C$ in the proof of 
Lemma~\ref{lem:ellipsoid_inclusion} has the form
\[
  \begin{bmatrix}
    E_{1,1} & d_{1,2} E_{1,2} & \cdots & d_{1,n} E_{1,n} & 
    \frac{a_1}{b_1} E_{1,n+1}\\
    d_{2,1} E_{2,1} & E_{2,2} & \ddots & \vdots & \vdots\\
    \vdots & \ddots & \ddots & d_{n-1,n} E_{n-1,n} & \vdots \\
    d_{n,1} E_{n,1} & \cdots & d_{n,n-1} E_{n,n-1} & E_{n,n} & 
    \frac{a_n}{b_n} E_{n,n+1} \\[+0,5ex]
    \frac{a_1}{b_1} E_{n+1,1} & \frac{a_2}{b_2} E_{n+1,2} & \cdots &
    \frac{a_n}{b_n} E_{n+1,n} & E_{n+1,n+1}
  \end{bmatrix} ,
\]
where $d_{ij} := \frac{a_i\, a_j}{b_i\, b_j}$.
\end{rem}

Now we prove exactness of the criterion for the containment of a ball in a
polyhedron.

\begin{lemma}\label{lem:ball-polytope_inclusion}
Let $S_A$ be a ball of radius $r > 0$ in normal form~\eqref{eq:ellipsoid},
and let $S_B$ be a polyhedron in normal form~\eqref{eq:polytope}.
For the containment of $S_A$ in $S_B$, system~\eqref{eq:inclusion} is necessary 
and sufficient.
\end{lemma}

\begin{proof}
Note first that $k = n+1$. 
Since the normal form $B(x)$ is monic, the linear polynomials describing $S_B$ 
are of the form $b_i(x) = 1 + \sum_{p=1}^n b_{i,p} x_p$ for $i=1,\ldots,l$. 
If $S_A$ is contained in the halfspace $b_i(x) \ge 0$, we have 
$\frac{1}{r^{2}} \geq \sum_{p=1}^n b_{i,p}^2$.

We give a feasible matrix $C$ to system \eqref{eq:inclusion} to show exactness
of the criterion. In this case, $C$ is an $(n+1)l \times (n+1)l$-block matrix
defined as follows:
\[
  \left(C_{i,j}\right)_{s,t} = 
  \begin{cases}
    \frac{r^{2} b_{s,i}^2}{2} & i=j < k,\ s=t ,\\
    1- \frac{r^{2}}{2} \sum_{p=1}^n b_{s,i}^2 & i=j=k,\ s=t ,\\
    \frac{ r^{2} b_{s,i} b_{s,j} }{2} & i < k,\ j<k, \ i \neq j,\ s=t , \\
    \frac{r\, b_{s,j}}{2} & i=k,\ j<k,\ s=t , \\
    \frac{r\, b_{s,i}}{2} & j=k,\ i<k, \ s=t , \\
    0 & \text{ otherwise.}
  \end{cases}
\]
To show positive semidefiniteness of $C$, consider a vector $x \in \R^{(n+1)l}$.
Decompose $x$ into blocks of length $l$, and we write $x_{i,j}$ for the $j$-th 
entry in the $i$-th block. Now $C$ is positive semidefinite since
\begin{align*}
  x^{T} C x = \sum_{s=1}^{l} & \Bigg[ \sum_{i=j<k} x_{i,s}^2 
    \frac{r^{2} b_{s,i}^2}{2}
  + x_{k,s}^2 \left( 1- \frac{r^{2}}{2} \sum_{p=1}^n b_{s,p}^2 \right) \\
  &  + 2 \sum_{i<j<k} x_{i,s} x_{j,s} \frac{ r^{2} b_{s,i} b_{s,j} }{2}
  + 2 \sum_{i<k} x_{i,s} x_{k,s} \frac{ r\, b_{s,i} }{2} \Bigg] \\
  = \sum_{s=1}^{l} &  \Bigg[ \left( \sum_{i=j<k} x_{i,s} 
    \frac{r\, b_{s,i}}{\sqrt{2}} 
  + \frac{x_{k,s}}{\sqrt{2}} \right)^2 + \frac{1}{2} x_{k,s}^2 \left( 1
  - r^{2} \sum_{p=1}^n b_{s,p}^2 \right) \Bigg] \geq 0
\end{align*}
for all $x\in\R^{(n+1)l}$.
The term $1 - r^{2} \sum_{p=1}^n b_{s,p}^2$ is non-negative since the ball of 
radius $r$ is contained in $S_B$ and therefore
$\frac{1}{r^{2}} \geq \sum_{p=1}^n b_{s,p}^2$. 
By construction, the sum of the diagonal blocks is the identity matrix $I_l$.
Every $B_p$ is a linear combination of only two blocks of $C$,
\[
  B_p = \frac{1}{r} C_{n+1,p} + \frac{1}{r} C_{p,n+1}.
\]
\end{proof}

Observe that in Lemma~\ref{lem:polytope_inclusion_neccessary},
\ref{lem:ellipsoid_inclusion} and~\ref{lem:ball-polytope_inclusion} for
rational input, $C$ is rational as well.

\section{Block diagonalization, transitivity and
containment of spectrahedra in polytopes}
\label{sec:spectrahedron-polytope-inclusion}

In \cite[Prop. 5.3]{Helton2010} Helton, Klep and McCullough showed that the
containment criterion \eqref{eq:inclusion} is exact in an important case,
namely if $S_B$ is the cube, given by the monic linear pencil
\begin{equation} \label{eq:r-cube}
  B(x) = I_{2n} +\frac{1}{r} \sum_{p=1}^{n} x_p \left(E_{pp}-E_{n+p,n+p}\right)
\, .
\end{equation}
The goal in this section is to generalize this to all polyhedra $S_B$ given in
normal form \eqref{eq:polytope}, not only for the original criterion, but also
for the variations discussed in Corollaries~\ref{cor:contain-positive} and
\ref{cor:contain_symmetry}. 
As in Lemma~\ref{lem:polytope_inclusion_neccessary}, in case that $S_B$ is 
unbounded we have to use the extended normal form $\widehat{A}(x)$ instead of 
$A(x)$.

\begin{thm}\label{thm:spec_polytope}
Let $A(x)\in\mathcal{S}_k[x]$ be a monic linear pencil with  extended linear
pencil  $\widehat{A}(x)\in \mathcal{S}_{k+1}[x]$ as defined in
equation~\eqref{eq:extended_spec} and let $B(x)\in\mathcal{S}_l[x]$ be the
normal form of a polyhedron.
When applied to the pencils $\widehat{A}(x)$ and $B(x)$
criterion~\eqref{eq:inclusion} is necessary and sufficient for the inclusion
$S_{\widehat{A}} = S_A \subseteq S_B$. If $S_B$ is a polytope then the pencil 
$A(x)$ can be used instead of $\widehat{A}(x)$.
\end{thm}

In order to prove this statement (where the sufficiency-parts are clear from
Theorem~\ref{thm:Helton_Contain}) we have to develop some auxiliary results on
the behavior of the criterion with regard to block diagonalization and
transitivity, which are also of independent interest.

As pointed out by an anonymous referee, this theorem can also be deduced from
results of Klep and Schweighofer in \cite{Klep2011}. A linear scalar-valued
polynomial is positive on a spectrahedron if and only if it is positive on the
matricial version of the spectrahedron. 

We use the following statement from \cite{Helton2010} on the block
diagonalization. As usual, for given matrices $M^1, \ldots, M^l$, we denote by
the direct sum $\bigoplus_{i=1}^l M^i$ the block matrix with diagonal blocks
$M^1, \ldots, M^l$ and zero otherwise.

\begin{prop}
\label{prop:direct_sum}
\cite[Proposition 4.2]{Helton2010}
Let $A(x) \in \mathcal{S}_k[x]$, $B(x) \in \mathcal{S}_l[x]$ and 
$D^q(x) \in \mathcal{S}_{d_q}[x]$ be linear pencils with 
$D^q (x) = D_0^q + \sum_{p=1}^n x_p D_p^q $, $q=1,\ldots,m$.

If $B(x) = \bigoplus_{q=1}^m D^q (x)$ is the direct sum with
$l = \sum_{q=1}^m d_q$, then system~\eqref{eq:inclusion} is feasible if and only
if for all $q = 1, \ldots, m$ there exists a $k d_q \times k d_q$-matrix $C^q$,
consisting of $k \times k$ blocks of size $d_q \times d_q$, such that
\begin{equation} \label{eq:b-block-inclusion}
  C^q = (C^q_{ij})_{i,j=1}^{k} \succeq 0, \quad
  \forall p = 0, \ldots, n : \ 
  D_p^q = \sum_{i,j=1}^{k} a_{ij}^p C_{ij}^q 
\end{equation}
is feasible.
\\
An analogous statement holds for criterion~\eqref{eq:inclusion_relaxed} and the
criteria discussed in Corollaries~\ref{cor:contain-positive} 
and~\ref{cor:contain_symmetry}.
\end{prop}

Since~\cite{Helton2010} does not contain a proof of this statement, we provide a
short one.

\begin{proof}
Let $C^ 1, \ldots, C^m$ be solutions to \eqref{eq:b-block-inclusion}, and set 
$C = \bigoplus_{q=1}^m C^q$. Define $C'$ as the direct sum of blocks of $C$, 
$C'_{ij} = \bigoplus_{q=1}^m C^q_{ij}$. Then $C'$ is a solution 
to~\eqref{eq:inclusion}:
$C'$ results by simultaneously permuting rows and columns of $C$ and is thus
positive semidefinite. We have 
$ B_p = \bigoplus_{q=1}^m D_p^q = \bigoplus_{q=1}^m \sum_{i,j=1}^k a_{ij}
C_{ij}^q = \sum_{i,j=1}^k a_{ij} C'_{ij}$.

Conversely, let $C'$ be a solution to~\eqref{eq:inclusion}. 
We are interested in the $m$ diagonal submatrices of each block $C'_{ij}$,
defined as follows: 
For $q \in \{1, \ldots, m \}$, let $C'^q_{ij}$ be the $d_q\times d_q$ submatrix
of $C'_{ij}$ with row and column indices 
$\{ \sum_{r=1}^{q-1}d_r+1 , \ldots, \sum_{r=1}^{q}d_r \}$.
Now the submatrix $C^q = (C'^q_{ij})_{i,j=1}^k$ consisting of the $q$-th
diagonal blocks of each matrix $C'_{ij}$ is a solution 
to~\eqref{eq:b-block-inclusion}.
$C^q$ is a principal submatrix of $C'$ and thus positive semidefinite. 
The equations in~\eqref{eq:b-block-inclusion} are a subset of the equations
in~\eqref{eq:inclusion} and remain valid.
\end{proof}

We now prove transitivity of the containment criterion. We begin with a simple
auxiliary lemma on the Kronecker products of corresponding blocks of block
matrices.

\begin{lemma} \label{lem:block-block-psd}
Let $A \succeq 0$ consist of $m \times m$ blocks of size $n_a \times n_a$ and 
$B \succeq 0$ consist of $m \times m$ blocks of size $n_b \times n_b$. 
Then $\sum_{s,t=1}^m (A_{st} \otimes B_{st}) \succeq 0$.
\end{lemma}

\begin{proof}
First note that we have 
$ v^T \left( \sum_{s,t=1}^m A_{st} \right) v 
  = (v^T \ldots v^T) A  (v^T \ldots v^T)^T \geq 0 $ for all $v \in \R^{n} $, 
as in the proof of Theorem~\ref{thm:Helton_Contain}, and hence 
$\sum_{s,t=1}^m A_{st} \succeq 0$.

Since $A, B \succeq 0$, we have $A \otimes B \succeq 0$ as well.
$\left(A_{st} \otimes B_{st} \right)_{s,t=1}^m$ is a principal submatrix of this
matrix and therefore also positive semidefinite. Summing up the blocks of this
matrix and applying our initial considerations, we see that 
$\sum_{s,t=1}^m A_{st} \otimes B_{st} \succeq 0$.
\end{proof}

The criteria from Theorem~\ref{thm:Helton_Contain} are transitive in the
following sense. 

\begin{thm}[Transitivity] \label{thm:Helton_transitive}
Let $D(x)\in\mathcal{S}_d[x]$, $E(x)\in\mathcal{S}_e[x]$ and
$F(x)\in\mathcal{S}_f[x]$ be linear pencils in $n$ variables.
If criterion~\eqref{eq:inclusion}, criterion~\eqref{eq:inclusion_relaxed} or
criterion~\eqref{eq:inclusion_positive} certifies the inclusion $S_D \subseteq
S_E$ and the inclusion $S_E \subseteq S_F$, it also certifies 
$S_D \subseteq S_F$. 
\end{thm}

The transitivity statement concerning \eqref{eq:inclusion} can be interpreted
from an operator theoretic point of view. It states the well-known fact that the
composition of two completely positive maps is again completely positive. 
Our approach enables us to extend the statement to the relaxed
criteria~\eqref{eq:inclusion_relaxed} and \eqref{eq:inclusion_positive}.

\begin{proof}
We first consider the relaxed version~\eqref{eq:inclusion_positive}. 
Let $C^{DE}$ be the $de \times de$-matrix certifying the inclusion 
$S_D \subseteq S_E$ and $C^{EF}$ the $ef \times ef$-matrix certifying the
inclusion $S_E \subseteq S_F$.
$C^{DE}$ consists of $d \times d$ block matrices of size $e \times e$, $C^{EF}$
consists of $e \times e$ block matrices of size $f \times f$. 

We prove that the matrix $C^{DF}$ consisting of $d \times d$ blocks of size $f
\times f$ and defined by
\[
  C^{DF}_{ij} := \sum_{s,t=1}^e (C^{DE}_{ij})_{st} C^{EF}_{st}
\]
is a solution to system~\eqref{eq:inclusion_relaxed} for the inclusion $S_D
\subseteq S_F$.

To show $C^{DF} \succeq 0$, we start from $C^{DE} \succeq 0$ and 
$C^{EF} \succeq 0$. Define a new matrix $\tilde{C}^{DE}$ by
$(\tilde{C}^{DE}_{st})_{ij} := (C^{DE}_{ij})_{st}$, 
permuting the rows and columns of $C^{DE}$. 
Since rows and columns are permuted simultaneously, positive semidefiniteness is
preserved. 
We think of $\tilde{C}^{DE}$ as having $e\times e$ blocks of size $d\times d$. 
$C^{DF}$ now simplifies to 
$C^{DF} = \sum_{s,t=1}^e \tilde{C}^{DE}_{st} \otimes C^{EF}_{st}$. 
Using Lemma~\ref{lem:block-block-psd}, $C^{DF} \succeq 0$ follows.

Next we show $F_p - \sum_{i,j=1}^{d} d^p_{ij} C^{DF}_{ij} \succeq 0$ for
$p=0,\ldots,n$. By assumption,
\begin{align}
  F_p - \sum_{i,j=1}^{e} e^p_{ij} C^{EF}_{ij} = G^{EF} \succeq 0 \text{ and }
  E_p - \sum_{i,j=1}^{d} d^p_{ij} C^{DE}_{ij} = G^{DE} \succeq 0 .
\label{eq:assump2}
\end{align}
By definition of $C^{DF}$ and the right equation of~\eqref{eq:assump2}, we have
\[
  \sum_{i,j=1}^d d^p_{ij} C^{DF}_{ij} 
  = \sum_{i,j=1}^{d} d^p_{ij} \sum_{s,t=1}^e \left(C_{ij}^{DE} \right)_{st}
    C^{EF}_{st} 
  = \sum_{s,t=1}^e \left( E_p -G^{DE} \right)_{st} C^{EF}_{st} , 
\]
and then positive semidefiniteness of $G^{EF}$ and $G^{DE}$ yield
\[
  F_p - \sum_{i,j=1}^d d^p_{ij} C^{DF}_{ij} 
  = G^{EF} + \sum_{s,t=1}^e G_{st}^{DE} C_{st}^{EF} \succeq 0 \, .
\]
The non-relaxed version~\eqref{eq:inclusion} as well as the relaxed
version~\eqref{eq:inclusion_relaxed} follow by choosing $G^{EF}$ and $G^{DE}$ 
in~\eqref{eq:assump2} to be zero matrices. 
\end{proof}

We can now establish the proof of Theorem~\ref{thm:spec_polytope},
which also completes the proof of Theorem~\ref{thm:summaryexactness}.

\begin{proof} (of Theorem~\ref{thm:spec_polytope}.)
Every (monic) linear pencil $B(x)$ in normal form~\eqref{eq:polytope} can be
stated as a direct sum
\[
  B(x) = \bigoplus_{q=1}^{l} b^q(x) = \bigoplus_{q=1}^{l} (\mathds{1}_l +Bx)_q .
\]
Therefore, Proposition~\ref{prop:direct_sum} implies that
system~\eqref{eq:inclusion} is feasible if and only if the system
\begin{equation*}
  C^q = (C_{ij})_{i,j=1}^k\succeq 0 ,\quad 1 
      = \sum_{i=1}^{k} C_{ii}^q ,\quad
      \forall p=1,\ldots,n:\ b^q_p = \sum_{i,j=1}^k a_{ij}^p C_{ij}^q 
\end{equation*}
is feasible for all $q=1,\ldots,l$. Note that $C^q$ is in $\mathcal{S}_k$.
Hence, the system has the form
\begin{equation}
  C^q = (C_{ij})_{i,j=1}^k\succeq 0, \quad 
    1 = \left\langle I_k, C^q \right\rangle, \quad 
    \forall p=1,\ldots,n:\ b^q_p = \left\langle A_p, C^q \right\rangle .
\label{eq:proof_spec_polytope}
\end{equation}
In the following, we show the existence of a solution by duality theory of
semidefinite programming and transitivity of~\eqref{eq:inclusion}, see 
Theorem~\ref{thm:Helton_transitive}.

Let $b_1^q, \ldots, b_n^q$ be the coefficients of the linear form 
$b^q(x) = (\mathds{1}_l + Bx)_q$. 
Since $B(x)$ is in normal form~\eqref{eq:polytope}, the vector 
$b^q := (b_1^q, \ldots, b_n^q)$ is an inner normal vector to the hyperplane
$b^q(x) = 0$. Consider the semidefinite program
\begin{align}
  r_q := \max &\ \langle -b^q, x\rangle \tag{$P_q$} \\
  \textnormal{s.t.} &\ A(x)\succeq 0 \notag
\end{align}
for all $q=1,\ldots,l$. By assumption, $(P_q)$ is strictly feasible and the
optimal value is finite. Hence (see, e.g., \cite[Thm. 2.2]{de-klerk-book}), 
the dual problem
\begin{align}
  \min &\ \langle I_k , Y^q \rangle \notag \\
  \textnormal{s.t.} &\ \left\langle A_p ,Y^q \right\rangle 
    = b^q_p \quad \forall p=1,\ldots,n \tag{$D_q$}, \\
  &\ \ Y^q \succeq 0 \notag
\end{align}
has the same optimal value and attains it. 
(Note that by duality $\langle -A_p,Y^q \rangle = - b^q_p$.) We can scale the
primal and dual problems simultaneously by dividing by $r_q$ and get
\begin{align}
  1 = \min & \ \langle I_k , \tilde{Y}^q \rangle \notag \\
  \textnormal{s.t.} &\ \langle A_p ,\tilde{Y}^q \rangle 
  = \frac{b^q_p}{r_q}\ \quad \forall p=1,\ldots,n, \tag{$\tilde{D}_q$} \\
  &\ \ \tilde{Y}^q \succeq 0 \notag ,
\end{align}
in which $\tilde{Y}^q := \frac{Y_q}{r_q}$.

Since in the dual $(D_q)$ the optimal value is attained, in $(\tilde{D}_q)$ it
is as well, i.e., for all $q=1,\ldots,l$ there exists a $k\times k$-matrix $C^q$ 
such that
\[
  C^q\succeq 0 ,\quad 
  1 = \left\langle I_k , C^q \right\rangle , \quad
  \frac{b^q_p}{r_q} = \left\langle A_p , C^q \right\rangle .
\]
As mentioned before~\eqref{eq:proof_spec_polytope}, the matrices $C^q$ certify 
the inclusion $S_A\subseteq S_{B'}$, where $B'(x)$ is defined as the scaled 
monic linear pencil 
\[
  B'(x) = \bigoplus_{q=1}^{l} \left( 1 +\sum \frac{b^q_p}{r_q} x_p \right) .
\]
Now we have to distinguish between the two cases in the statement of the 
theorem.
 
First consider the case where $S_B$ is a polytope. 
Since $B(x)$ is in normal form, we have
$\max_{x\in S_B} \langle -b^q, x\rangle = 1$.
Further, since $S_A \subseteq S_B$, the definition of $r_q$ implies
$ r_q \leq 1$ and hence, $S_{B'} \subseteq S_{B}$. 
By transitivity and by exactness of the criterion for polytopes, see
Theorem~\ref{thm:Helton_transitive} and
Lemma~\ref{lem:polytope_inclusion_neccessary},
respectively, there is a solution of system~\eqref{eq:inclusion} that certifies
$S_A\subseteq S_B$.

To prove the unbounded case in the theorem, we construct a solution
to~\eqref{eq:inclusion} for the inclusion 
$S_{\widehat{A}}\subseteq S_{\widehat{B'}}$, where
$\widehat{B'}(x) =  1 \oplus B'(x)$ denotes the extended normal
form~\eqref{eq:extended_spec} of the polyhedron $S_{B'}$. Then the claim follows 
by Lemma~\ref{lem:polytope_inclusion_neccessary} and
Theorem~\ref{thm:Helton_transitive}, as above, since
$S_{\widehat{B'}} \subseteq S_{B'}$ is certified by \eqref{eq:inclusion}.

First note that $S_{\widehat{A}} \subseteq S_B$ is equivalent to 
$S_A \subseteq S_B$. 
Denote by $C'$ the matrix that certifies the inclusion $S_A\subseteq S_{B'}$.
Then the symmetric $(k+1)(l+1)\times(k+1)(l+1)$-matrix
\[
  \widehat{C} := E_{11} \oplus 
    \left[\begin{array}{cc}	0 & 0 \\
				0 & C'_{ij}\end{array}\right]_{i,j=1}^k ,
\]
where $E_{11}$ and the blocks \begin{footnotesize} 
$\Big[\begin{array}{cc}	0 & 0 \\
  0 & C'_{ij}\end{array}\Big]_{i,j=1}^k$
\end{footnotesize}
are of size $(l+1) \times (l+1)$,
certifies the inclusion $S_{\widehat{A}}\subseteq S_{\widehat{B'}}$. Indeed,
adding
zero-columns and zero-rows simultaneously preserves positive semidefiniteness
and, clearly, the sum of the diagonal blocks of $\widehat{C}$ is the identity
matrix
$I_{l+1}$. 
Since in every $\widehat{A}_p$ the first column and the first row are the zero
vector,
we get
\[ 
  \sum_{i,j=0}^{k} \widehat{a}^{\, p}_{ij}\, \widehat{C}_{ij}
  = 0 \cdot E_{11}
  + \begin{bmatrix}
    0 & 0 \\
    0 & \sum_{i,j=1}^{k} a^p_{ij} C'_{ij}
  \end{bmatrix}
  = \widehat{B'}_p\, ,
 \]
where $\widehat{a}^{\, p}_{ij}$ is the $(i,j)$-th entry of $\widehat{A}_p$.

Feasibility of the relaxed criteria is again implied by the feasibility of
\eqref{eq:inclusion}.

\end{proof}

\section{Containment of scaled spectrahedra and inexact cases}
\label{sec:scaled}

Contrasting the results of Sections~\ref{sec:criterion}
and~\ref{sec:spectrahedron-polytope-inclusion}, we first consider a situation
where the containment criterion fails and the relaxed
version~\eqref{eq:inclusion_positive} is strictly stronger.
In particular, this raises the question whether (for a spectrahedron $S_A$
contained in a spectrahedron $S_B$) the criterion becomes satisfied when scaling
$S_A$ by a suitable factor. 
In Proposition~\ref{prop:scaling}, we answer this question in the affirmative.
We then close the paper by applying this result on optimization versions of the
containment problem.

\subsection{Cases where the criterion fails}
\label{sec:non-exact}

We review an example from~\cite[Example 3.1, 3.4]{Helton2010} which shows that 
the containment criterion is not exact in general. 
We then contrast this phenomenon by showing that for this example there exists 
a scaling factor $r$ for one of the spectrahedra so that the containment 
criterion is satisfied after this scaling. 

Consider the monic linear pencils 
$A(x) = I_3 + x_1 (E_{1,3}+E_{3,1}) + x_2 (E_{2,3}+E_{3,2}) \in \mathcal{S}_3[x]$ 
and 
\[
  B(x) =
\begin{bmatrix}
  1 & \\
  & 1 \\
\end{bmatrix} + x_1 
\begin{bmatrix}
  1 & \\
  & -1 \\
\end{bmatrix} + x_2 
\begin{bmatrix}
  & 1 \\
  1 & \\
\end{bmatrix} .
\]
Clearly, both define the unit disc, that is $S_A = S_B$.

\medskip

\noindent
\begin{claim}
The containment question $S_B \subseteq S_A$ is certified by
criterion~\eqref{eq:inclusion}, while the reverse containment question
$S_A \subseteq S_B$ is not certified by the criterion.
\end{claim}

\medskip

First, we look into the inclusion $S_B \subseteq S_A$ (where the roles of $A$
and $B$ in \eqref{eq:inclusion} have to be interchanged).
Criterion~\eqref{eq:inclusion} is satisfied if and only if there exist 
$c_1, c_2, c_3 \in \R$ such that
\[
  C = \left[
\begin{array}{ccc|ccc}
  \frac{1}{2} & 0 & \frac{1}{2} & 0 & c_1 & c_2 \\
  0 & \frac{1}{2} & 0 & -c_1 & 0 & c_3 \\
  \frac{1}{2} & 0 & \frac{1}{2} & -c_2 & 1-c_3 & 0 \\[0.5ex] \hline
          &      &             &      &       &   \\ [-2ex]
  0 & -c_1 & -c_2 & \frac{1}{2} & 0 & -\frac{1}{2} \\
  c_1 & 0 & 1-c_3 & 0 & \frac{1}{2} & 0 \\
  c_2 & c_3 & 0 & -\frac{1}{2} & 0 & \frac{1}{2} \\
\end{array} 
\right] \in \R^{6\times 6}
\]
is positive semidefinite.
Since the $2\times2$-block in the top left corner is positive definite, the
matrix $C$ is positive semidefinite if and only if the Schur complement with
respect to this block is positive semidefinite. 
One can easily check that this is the case if and only if $c_1=c_3=\frac{1}{2}$
and $c_2=0$.

Conversely, $S_A \subseteq S_B$ is certified by~\eqref{eq:inclusion} if and
only if there exist $c_1, \ldots, c_{12} \in \R$ such that
\[
  C =
\left[ 
\begin{array}{cc|cc|cc}
  c_1 & c_2 & c_9 & c_{10} & \frac{1}{2} & c_7 \\
  c_2 & c_3 & c_{11} & c_{12} & -c_7 & -\frac{1}{2} \\ [0.5ex] \hline
  c_9 & c_{11} & c_{4} & c_{5} & 0 & c_8 \\
  c_{10} & c_{12} & c_5 & c_6 & 1-c_8 & 0 \\ \hline
         &       &     &     &       &   \\ [-2ex]
  \frac{1}{2} & -c_7 & 0 & 1-c_8 & 1-c_1-c_4 & -c_2-c_5 \\
  c_7 & -\frac{1}{2} & c_8 & 0 & -c_2-c_5 & 1-c_3-c_6 \\
\end{array}
\right] \in\R^{6\times 6}
\]
is positive semidefinite. 
We show the infeasibility of the system~\eqref{eq:inclusion}.

Assume that $C$ is positive semidefinite. Then all principal minors are
non-negative. Consider the principal minor
\[
  \begin{vmatrix}
    c_1 & \frac{1}{2} \\
    \frac{1}{2} & 1-c_1-c_4 
  \end{vmatrix}
  = c_1(1 - c_1 - c_4) -\frac{1}{4}
  = \left[ c_1(1 - c_1) - \frac{1}{4} \right] - c_1c_4 .
\]
Since the expression in the brackets as well as the second term are always less
than or equal to zero the minor is non-positive.
Therefore, $c_1(1-c_1)-\frac{1}{4} = 0$ and $c_1 c_4 = 0$, or equivalently,
$c_1=\frac{1}{2}$ and $c_4=0$.

Recall that whenever a diagonal element of a positive semidefinite matrix is
zero, the corresponding row is the zero vector, that is 
$c_5 = c_8 = c_9 = c_{11} = 0$. 
Now, we get a contradiction since the principal minor
\[
  \begin{vmatrix} c_6 & 1 - c_8 \\ 1 - c_8 & 1 - c_1 - c_4 \end{vmatrix}
  = \begin{vmatrix} c_6 & 1 \\ 1 & \frac{1}{2} \end{vmatrix}
  = \frac{1}{2}\, c_6 - 1
\]
implies that $c_6\geq2$ and therefore  $ 1 - c_3 - c_6 \leq -1 - c_3 < 0 $ or
$ c_3 < -1 $. This proves the claim.
\newline

Now, generalizing $A(x)$, let $A^r(x)$ be the linear pencil of the ball
with radius $(1>)r>0$ in normal form.
With regard to the containment question $S_{A^r} = r S_A \subseteq S_B$, we show
the feasibility of system~\eqref{eq:inclusion} for $r$ sufficiently small.
Consider the matrix
\[
  C = \left[ 
\begin{array}{cc|cc|cc}
  c & 0 & 0 & c & \frac{r}{2} & 0 \\
  0 & c & -c & 0 & 0 & -\frac{r}{2} \\ [0.5ex] \hline
  0 & -c & c & 0 & 0 & \frac{r}{2} \\
  c & 0 & 0 & c & \frac{r}{2} & 0 \\ [0.5ex] \hline
  \frac{r}{2} & 0 & 0 & \frac{r}{2} & 1-2 c & 0 \\
  0 & -\frac{r}{2} & \frac{r}{2} & 0 & 0 & 1-2 c \\
\end{array}
\right] \in\R^{6\times 6}.
\]
Obviously the equality constraints in \eqref{eq:inclusion} are fulfilled.

As above, if $c=0$ or $1-2c=0$, then $r=0$. Therefore, $0 < c < \frac{1}{2}$
and the $2\times2$-block in the top left corner $C_{11}$ is positive
definite. Thus the matrix $C$ is positive semidefinite if and only if the Schur
complement with respect to $C_{11}$ is positive semidefinite. 
This is the case if and only if
\[
  1-2c-\frac{r^2}{4c} \geq 0
  \ \Leftrightarrow\ f(c) := 8 c^2 - 4 c + r^2 \leq 0 .
\]
Assume $r > \frac{1}{2}\sqrt{2}$. 
Then $f(c)>0$ for all $c$ since $f$ has no real roots and the constant term
$f(0) = r^2$ is positive.
Otherwise, $ f(\frac{1}{4}) = -\frac{1}{2} + r^2 \leq 0 $.
Hence, system~\eqref{eq:inclusion} is feasible for 
$0 < r \leq \frac{1}{2}\sqrt{2}$.

The problem of maximizing $r$ such that the system~\eqref{eq:inclusion} is
feasible can be formulated as a semidefinite program. A numerical computation
yields an optimal value of $0.707 \approx \frac{1}{2}\sqrt{2}$.

Note that we are in the situation of Corollary~\ref{cor:contain_symmetry}.
For the relaxed version~\eqref{eq:inclusion_positive}, a numerical computation
gives the optimal value of $0.950 \approx \frac{19}{20}$.
In particular, this shows that the relaxed
criterion~\eqref{eq:inclusion_positive} can be satisfied in cases where the
non-relaxed criterion~\eqref{eq:inclusion} does not certify an inclusion.

It is an open research question to establish a quantitative relationship
comparing criterion~\eqref{eq:inclusion_positive} to \eqref{eq:inclusion} in the
general case.

\subsection{Containment of scaled spectrahedra}

For a monic linear pencil $A(x) \in \mathcal{S}_k[x]$ 
and a constant $\nu>0$ define
\begin{equation}
  A^{\nu}(x) := A\left(\frac{x}{\nu}\right) 
  = I_k + \frac{1}{\nu} \sum_{p=1}^{n} x_p\, A_p ,
\label{eq:scaled_pencil}
\end{equation}
the {\it $\nu$-scaled (monic linear) pencil}. 
Similarly, we denote by $\nu S_{A} := \{ x\in\R^n \, : \, A^{\nu}(x)\succeq0
\} $ the corresponding {\it $\nu$-scaled spectrahedron}.

Generalizing the observation from Section~\ref{sec:non-exact}, we show that for
two spectrahedra $S_A$ and $S_B$, containing the origin in their interior, there
always exists some scaling factor $\nu$ such that the
criteria~\eqref{eq:inclusion} and~\eqref{eq:inclusion_relaxed} certify the
inclusion $\nu S_A\subseteq S_B$. 
This extends the following result of Ben-Tal and Nemirovski, who treated
containment of a cube in a spectrahedron (in which case they can even give a
bound on the scaling factor).

\begin{prop}
\cite[Thm. 2.1]{ben-tal-nemirovski-2002}
Let $S_A$ be the cube~\eqref{eq:r-cube} with edge length $r>0$ and consider
a monic linear pencil $B(x)$. Let $\mu = \max_{p=1,\ldots,n} \rank B_p$. 
If $S_A \subseteq S_B$, then system~\eqref{eq:inclusion} is feasible for the
$\nu(\mu)$-scaled cube ${\nu(\mu)} S_{A}$, where $\nu(\mu)$ is given by
\[
  \nu(\mu) = \min_{y\in\R^{\mu}, \|y\|_1 = 1} 
  \left\{ 
  \int_{\R^{\mu}} \left| \sum_{i=1}^{\mu} y_i u_i^2 \right| 
  \left( \frac{1}{2\mu} \right)^{\frac{\mu}{2}} 
  \exp\left( -\frac{u^{T} u}{2} \right) du
  \right\} .
\]
  \item For all $\mu$ the bound $\nu(\mu) \geq
\frac{2}{\pi\sqrt{\mu}}$
holds.
\label{prop:scaled_cube}
\end{prop}

A quantitative result as presented in the last Proposition is not known for the
general case. However, combining Proposition~\ref{prop:scaled_cube} with our
results from Sections~\ref{sec:criterion} 
and~\ref{sec:spectrahedron-polytope-inclusion} we get that for spectrahedra with 
non-empty interior, there is always a scaling factor such that 
system~\eqref{eq:inclusion} and thus also system~\eqref{eq:inclusion_relaxed} hold.

\begin{prop} \label{prop:scaling}
Let $A(x)$ and $B(x)$ be monic linear pencils such that $S_A$ is bounded.
Then there exists a constant $\nu>0$ such that for the scaled spectrahedron 
$\nu S_{A}$ the inclusion $\nu S_A \subseteq S_B$ is certified by
the systems~\eqref{eq:inclusion} and~\eqref{eq:inclusion_relaxed}.
\end{prop}

We provide a proof based on the framework established in the previous
sections.
Alternatively it can be deduced from statements about the matricial relaxation
of criterion~\eqref{eq:inclusion} given in the work by Helton and
McCullough~\cite{Helton2004}, see also~\cite{Helton2010}.
Criterion~\eqref{eq:inclusion} is satisfied for linear pencils $A^{\nu}(x)$ and
$B(x)$ if and only if the matricial version of $\nu S_A$ is contained in the
matricial version of $S_B$.

\begin{proof}
Denote by $S_D$ the cube, defined by the monic linear pencil~\eqref{eq:r-cube}, 
with the minimal edge length such that $S_A$ is contained in it, which can be
computed by a semidefinite program, see Theorem~\ref{thm:spec_polytope}. 
Since $B(x)$ is monic, there is an open subset around the origin contained in
$S_B$. Thus there is a scaling factor $\nu_1>0$ so that 
$\nu_1 S_A \subseteq \nu_1S_D \subseteq S_B$. 

By Proposition~\ref{prop:scaled_cube}, there exists a constant $\nu_2>0$ such
that for the problem $\nu_2 \nu_1 S_{D} \subseteq S_B$ 
system~\eqref{eq:inclusion} has a solution $C^{D^{\nu}B }$ with $\nu = \nu_1 \nu_2$.
By Theorem~\ref{thm:spec_polytope}, there is a matrix $C^{A^{\nu} D^{\nu}}$
which solves \eqref{eq:inclusion} for the problem $\nu S_{A}$ in $\nu S_{D}$.

Finally, Theorem~\ref{thm:Helton_transitive} implies the feasibility of
system~\eqref{eq:inclusion} with respect to $\nu S_{A}$ and $S_B$ by the matrix
$C^{A^{\nu} B}$, as defined there.
\end{proof}

In the proof of Proposition~\ref{prop:scaling}, we scaled the spectrahedron
$S_A$ by a certain factor $\nu$.
Since $\nu S_A \subseteq S_B$ is equivalent to $S_A \subseteq \frac{1}{\nu} S_B$,
the criterion~\eqref{eq:inclusion} remains a positive semidefinite condition 
even in the presence of the factor $\nu$. 
Moreover, we can optimize for $\nu$ such that the criterion remains satisfied.
Proposition~\ref{prop:scaling} implies that for bounded spectrahedra represented
by monic linear pencils the maximization problem for $\nu$ always has a positive 
optimal value. 

This yields a natural framework for the approximation of smallest enclosing 
spectrahedra and largest enclosed spectrahedra. 
In~\cite[Section 4]{Helton2010}, the example of computing a bound for the norm
of the elements of a spectrahedron $S_A$ (represented by a monic linear pencil)
is given. 
This can be achieved by choosing $S_B$ to be the ball centered at the origin,
see~\eqref{eq:ellipsoid}.

As we have seen in Section~\ref{sec:non-exact}, applying 
criterion~\eqref{eq:inclusion_positive} to the problem is stronger than 
specializing criterion~\eqref{eq:inclusion} to it. 
However, for the criterion~\eqref{eq:inclusion}, we obtain a particularly nice 
representation,  it reduces to the semidefinite system

\begin{align}\label{eq:matricial_circumradius}
\begin{split}
  C &= \left( C_{ij} \right)_{i,j=1}^{k} \succeq 0, \\ 
  I_{n+1} &= \sum_{i=1}^{k} C_{ii}, \\ 
  \forall p = 1,\ldots,n,\ \forall\, (s,t)&\in\left\{ 1,\ldots,n+1 \right\}^{2}: \\
  \left( \sum_{i,j=1}^{k} a^p_{ij} C_{ij} \right)_{st} &=
  \begin{cases}
    \frac{1}{r} & \text{if } (s,t)\in\left\{ (p,n+1),\,(n+1,p) \right\}, \\
    0 & \text{else.}
  \end{cases}
\end{split} 
\end{align}

\bigskip
\begin{ack}
We would like to thank the anonymous referees for careful reading, 
detailed comments and additional relevant references.
\end{ack}

\newpage
\bibliography{bibcontainment}
\bibliographystyle{plain}
\end{document}